\documentclass[11pt,letterpaper]{amsart}
\usepackage{silence}
\usepackage{subfiles}
\usepackage[T1]{fontenc}
\usepackage[utf8]{inputenc}
\usepackage[dvipsnames]{xcolor}
\usepackage[yyyymmdd,hhmmss]{datetime}
\usepackage[numbers,sort,compress]{natbib}
\usepackage{amstext,amsthm,amssymb}
\usepackage{xparse,mathrsfs,mathtools}
\usepackage[british]{babel}
\usepackage[babel=true,english=british]{csquotes}
\usepackage{xr-hyper}
\usepackage{stmaryrd}
\usepackage{enumitem}
\usepackage[pdftex,unicode=true,%
pdfusetitle,pdfdisplaydoctitle=true,%
pdfpagemode=UseOutlines,%
bookmarks=true,bookmarksnumbered=false,bookmarksopen=true,bookmarksopenlevel=2,%
breaklinks=true,%
pdfencoding=unicode,psdextra,
pdfcreator={},
pdfborder={0 0 0},
colorlinks=false
]{hyperref}
\usepackage{lmodern,microtype}
\usepackage{geometry}
\usepackage[capitalise,nameinlink]{cleveref}
\usepackage[cal=zapfc,bb=ams,frak=esstix,scr=boondoxo]{mathalfa}

\hypersetup{
	colorlinks=false,
	pdfborder={0 0 0},
	pdfborderstyle={/S/U/W 0},
}

\newcommand*{\doi}[1]{\href{http://dx.doi.org/#1}{doi: {\texttt{\footnotesize\color{teal}#1}}}}

\newtheorem*{thm*}{Theorem}
\newtheorem{thm}{Theorem}[section]

\newtheorem{cor}[thm]{Corollary}

\theoremstyle{remark}
\newtheorem{rem}[thm]{Remark}
\newtheorem*{rem*}{Remark}

\numberwithin{equation}{section}
\crefformat{equation}{#2(#1)#3}
\crefformat{enumi}{#2(#1)#3}
\crefname{section}{\textsection}{\textsection\textsection}
\DeclareMathOperator{\meas}{meas}

\DeclarePairedDelimiter\parentheses{\lparen}{\rparen}
\DeclarePairedDelimiter\braces{\lbrace}{\rbrace}
\DeclarePairedDelimiter\brackets{\lbrack}{\rbrack}
\DeclarePairedDelimiter\abs{\lvert}{\rvert}
\DeclarePairedDelimiter\norm{\lVert}{\rVert}
\DeclarePairedDelimiter\floor{\lfloor}{\rfloor}

\NewDocumentCommand\set{ s o m o }{%
	\IfBooleanTF{#1}{\IfNoValueTF{#4}{\braces*{#3}}{\braces*{\,#3:#4\,}}}{%
		\IfNoValueTF{#2}{\IfNoValueTF{#4}{\braces{#3}}{\braces{\,#3:#4\,}}}{%
			\IfNoValueTF{#4}{\braces[#2]{#3}}{\braces[#2]{\,#3:#4\,}}}}%
}
\NewDocumentCommand\e{ s O{} m }{%
	\IfBooleanTF{#1}{%
		\operatorname{e}_{#2}\parentheses*{#3}%
	}{\operatorname{e}_{#2}\parentheses{#3}}%
}

\makeatother

\title[Additivity of resonance method]{An additive application of the resonance method}

\subjclass[2020]{% See http://www.ams.org/msc/msc2020.html
	Primary
	11L03; Secondary 11P99, 11N99}


\author{Athanasios Sourmelidis}
\address{%
	Athanasios~Sourmelidis\\%
		Univ. Lille, CNRS\\%
	UMR 8524 - Laboratoire Paul Painlevé\\%
	F-59000~Lille\\%
	France%
}
\email{athanasios.sourmelidis@univ-lille.fr}

\begin{document}
	\begin{abstract}
		We improve upon an Omega result due to Soundararajan with respect to general trigonometric polynomials having positive Fourier coefficients.
		Instead of Dirichlet's  approximation theorem we employ the resonance method and
		this leads to better extreme results in lattice point problems such as Dirichlet's divisor problem and Gauss' circle problem.
		Moreover, the present approach shows that the resonance method can also be viewed as an additive device, which has been mainly used in multiplicative problems so far.
		Its extension to trigonometric polynomials with complex coefficients is also discussed and its connection to Bohr and Jessen's proof of Kronecker's theorem is highlighted.
	\end{abstract}
	\maketitle
	%%%%%%%%%%%%%%%%%%%%%%%%%%%%%%%%%%%%%%%%%%%%%%%%%%%%%%%%%%
	%%%%%%%%%%%%%%%%%%%%%%%%%%%%%%%%%%%%%%%%%%%%%%%%%%%%%%%%%%
	%%%%%%%%%%%%%%%%%%%%%%%%%%%%%%%%%%%%%%%%%%%%%%%%%%%%%%%%%%
	\section{Introduction and main results}
	For two sequences of real numbers $0\leq\lambda_1<\lambda_2<\dots$ and $f(n)\geq0$, $n\geq1$,
	with $\sum_{n\geq1}f(n)<\infty$ we consider the trigonometric series
	\begin{align}\label{eq:trigonometric_polynomial}
	F(x):=\sum_{n\geq1}f(n)\e{\lambda_n x},\quad\mathrm{e}(x):=e^{2\pi ix},\quad x\in\mathbb{R}.
	\end{align}
	In 2003 Soundararajan \cite[Lemma 2.1]{Soundararajan_2003} showed that, for any integers $L,M,N\geq2$,  any set $\mathcal{M}\subseteq\mathbb{N}$ with cardinality $M$ such that $2|\lambda_m-\lambda_N|\leq{\lambda_N}$, $m\in\mathcal{M}$, any $\beta\in\mathbb{R}$ and any $X\geq2$,
	\begin{align}\label{eq:motivation}
			\max_{x\in\brackets*{X/2,(6L)^{M+1}X}}\abs*{\Re\parentheses*{e^{i\beta}F(x)}}\geq\frac{1}{8}\sum_{\substack{m\in\mathcal{M}}}f(m)-\frac{4F(0)}{\pi^2X\lambda_N}-\frac{1}{L-1}\sum_{\substack{\lambda_n\leq2\lambda_N}}f(n).
	\end{align}
	He then employed the above inequality, Vorono\"i's summation formula and certain information on the anatomy of integers to prove extreme results for Dirichlet's divisor problem and Gauss' circle problem.
	That is, if
	\[
	\Delta(x):=\sum_{n\leq x}d(n)-x\log x-(2\gamma-1)x\quad\text{and}\quad P(x):=\sum_{n\leq x}r(n)-\pi x,\quad x\geq1,
	\]
	where $d(n)$ denotes the number of divisors of $n$ and $r(n)$ the number of ways of writing $n$ as the sum of two integer squares, then \cite[Theorem 1.1]{Soundararajan_2003} 
	\begin{align}\label{eq:circle_and_divisor}
		\begin{split}
		\Delta(x)&=\Omega\parentheses*{x^{1/4}(\log x)^{1/4}(\log\log x)^{(3/4)(2^{4/3}-1)}(\log\log\log x)^{-5/8}},\\
		P(x)&=\Omega\parentheses*{x^{1/4}(\log x)^{1/4}(\log\log x)^{(3/4)(2^{1/3}-1)}(\log\log\log x)^{-5/8}}.
		\end{split}
	\end{align}
	
	Before we go any further, let us recall some notation.
	For a real-valued function $f$ and a positive function $g$, the {\it Omega} symbol $f=\Omega(g)$ means that $\limsup_{x\to\infty}|f(x)|/g(x)>0$, which is the negation of $f=o(g)$.
	We write $f=\Omega_+(g)$ if $\limsup_{x\to\infty}f(x)/g(x)>0$, and $f=\Omega_-(g)$ if $\liminf_{x\to\infty}f(x)/g(x)<0$.
	The Landau symbol $f=O(g)$ and the Vinogradov symbol $f\ll g$ have their usual meaning with the implicit constant being absolute unless indicated otherwise. 
	Lastly, if $f$ is positive, then the notation $f\asymp g$ stands for $g\ll f\ll g$.
	
	The $\Omega$-results on $\Delta(x)$ and $P(x)$ are the best known in a line of research initiated by Hardy \cite{Hardy1917}  and including most notably the work of Corr\'adi and K\'atai \cite{Corradi1967} and Hafner \cite{Hafner1981}.
	In particular, Hafner obtains $\Omega_+$- and $\Omega_-$-results for $\Delta(x)$ and $P(x)$, respectively, with slightly smaller exponents on the $\log\log x$ and observes that this is the limit of his method.
	On the other hand, Soundararajan's approach has its limit rooted in the inequality \eqref{eq:motivation} and specifically in the last term on the right-hand side. 
	For applications one is led to consider the parameter $L$ as an increasing function of $X$ in order for this term to be negligible.
	At the same time we aim to maximize $M$ so that the first term in the right-hand side of \eqref{eq:motivation} is as large as possible while keeping the interval $[X/2,(6L)^{M+1}X]$ of polynomial length.
	As a consequence the parameter $M$ can be at most $o(\log X)$ and relaxing this restriction has been the motivation for this work.
	\begin{thm}\label{thm:main_threorem}
{Let $F(x)$ be a trigonometric polynomial defined as in \eqref{eq:trigonometric_polynomial} and having, in particular, positive Fourier coefficients. 
Then} there is an explicitly computed absolute constant $C>0$ such that for any integers $M,N\geq1$, any set $\mathcal{M}\subseteq\mathbb{N}$ with cardinality $M$ for which $2|\lambda_m-\lambda_N|\leq{\lambda_N}$, $m\in\mathcal{M}$, any $\beta\in\mathbb{R}$ and any $T> Y\geq1$, we have that
		\begin{align*}
			\max_{x\in\brackets*{Y/2,2T}}\abs*{\Re\parentheses*{ e^{{i\beta}}F(x)}}\geq C \sum_{\substack{m\in\mathcal{M}}}f(m)-\frac{4F(0)}{\pi^2Y\lambda_N}+O\parentheses*{\frac{2^MY\log T}{T}\sum_{\lambda_n\leq2\lambda_{N}}f(n)}.
		\end{align*}
		\end{thm}
	Comparing the above inequality with \eqref{eq:motivation}, we find that they differ essentially in the last terms of their respective right-hand sides.
	The above error term yields the improvement on the magnitude of $M$ since the choice of $Y\ll T^{1-\epsilon}$ and $M\ll\epsilon\log T$ is admissible and the factor $2^MYT^{-1}\log T$ acts like ${(L-1)}^{-1}$.
	We can therefore take $M\asymp\log T$ while keeping the interval, where the extremal behavior of $F(x)$ is detected, of polynomial length.
		
	In contrast to the aforementioned results of Hardy, Hafner and Soundararajan who had employed Dirichlet's simultaneous  approximation theorem, 
	 Theorem \ref{thm:main_threorem} will follow by an application of the {\it resonance method}.
	Incidentally, this technique has been developed by Soundararajan \cite{Soundararajan2008} a few years after \cite{Soundararajan_2003} to address questions related to the extreme values of the Riemann zeta-function $\zeta(s)$, $s:=\sigma+it$, on the critical line\footnote{Similar results and the development of the resonance method inside the strip $1/2 < \sigma < 1$ have been carried out earlier -yet passed unknown for decades- by Voronin \cite{Voronin1988}.}
	 and of families of Dirichlet $L$-functions at $\sigma=1/2$.
	The idea is, roughly speaking, to prove $\Omega$-results for a Dirichlet polynomial (with multiplicative coefficients) $F(t)=\sum_{n\leq N}f(n)n^{it}$, by introducing another Dirichlet polynomial $R(t)$ that will {\it resonate} with $F(t)$ in the following sense:
\begin{align*}
\max_{t\in[Y,T]}|F(t)|\int_Y^T|R(t)|^2\mathrm{d}t\geq\abs*{\int_Y^TF(t)|R(t)|^2\mathrm{d}t}\geq A\int_Y^T|R(t)|^2\mathrm{d}t,\quad A>0.
\end{align*}
Then it is a matter of properly constructing $R(t)$ to maximize $A$ while satisfying the above inequalities.
Of course several error terms are involved in the background but when successfully estimated, the method yields also lower bounds for the measure of $t\in[Y,T]$ at which $|F(t)|$ is large.
	
We will first extend the resonance method to the setting of general trigonometric series $F(x)$ having positive Fourier coefficients. 
This assumption not only simplifies the exposition but it also produces superior results compared to when $f(n)\in\mathbb{C}$.
Now the resonator $R(x)$ will not be a Dirichlet polynomial but a trigonometric polynomial whose frequencies are linear combinations of the ``building blocks'' $\lambda_{m}$, $m\in\mathcal{M}$.
The proof of Theorem \ref{thm:main_threorem} distinctly illustrates how $R(x)$ works to bring forth the terms  $f({m})$.
In the special case when $\beta=0$ $(\bmod\,2\pi)$ we can also refine the theorem and obtain lower bounds for the measure of those $x\in[Y,T]$ at which $\Re(F(x))$ is large.
\begin{thm}\label{thm:main_subthreorem}
	{Let $F(x)$ be a trigonometric polynomial defined as in \eqref{eq:trigonometric_polynomial} and having, in particular, positive Fourier coefficients. 
		Then} for any integer $M\geq1$, any set $\mathcal{M}\subseteq\mathbb{N}$ with cardinality $M$ and any $T> Y\geq1$, we have that
	\begin{align*}
		\max_{x\in\brackets*{Y,T}}\Re\parentheses*{F(x)}\geq 2C \sum_{m\in\mathcal{M}}f(m)+O\parentheses*{\frac{2^MY\log T}{T}\sum_{n\geq1}f(n)}=:\Sigma,
		\end{align*}
		where $C>0$ is the constant from Theorem \ref{thm:main_threorem}.
If $\Sigma>0$, then for any $0<V\leq\Sigma$
		\[
		\mathrm{meas}\set*{x\in\brackets*{Y,T}:\Re\parentheses*{F(x)}\geq V}\geq\frac{(\Sigma-V) T}{2^{M}
				\log T\max_{x\in\brackets*{Y,{T}}}|F(x)|}.
		\]
\end{thm}
In the absence of any information regarding the sequence $f(n)$, $n\geq1$, (other than its positivity) we remark that the lower bound $\sum_{m\in\mathcal{M}}f(m)$ with $\#\mathcal{M}\asymp\log T$ is the limit of the method up to constants\footnote{For the sake of simplicity it was not pursued to get the optimal constants.
	The resonator constructed in the sequel is modeled from \cite{Aistleitner_2015} and it consist of terms with ``square-free'' frequencies (we will give a precise meaning of this).
	More involved resonators can produce {larger} constants.
	For instance a resonator modeled from \cite{Aistleitner2017a} would give larger explicitly computed constant $C>0$ in Theorem \ref{thm:main_threorem} and Theorem \ref{thm:main_subthreorem}.}.
On the one hand, we are free to choose any terms from $F(x)$ with the natural choice being the first terms of the descending sequence $\max_{n\geq1}f(n)=f(m_0)\geq f(m_1)\geq f(m_2)\geq\dots$, the order of which is often not easy to determine.
The distribution of the frequencies $\lambda_n$ on the real line has no effect on this selection, which is the advantage of working with positive coefficients.
On the other hand,  we can only consider at most $O(\log T)$ of these terms because of the factor $2^M$, which represents in this case the approximate number of terms that the resonator suffices to consist of in order to ``capture'' $M$ terms from $F(x)$.
Subject to these restrictions, we can obtain by implementing Theorem \ref{thm:main_threorem} to the proof of \cite[Theorem 1.1]{Soundararajan_2003} the following corollary.
	\begin{cor}\label{cor:improved_circle_and_divisor}
	We have that
	\begin{align*}
		\Delta(x)&=\Omega\parentheses*{x^{1/4}(\log x)^{1/4}(\log\log x)^{(3/4)(2^{4/3}-1)}(\log\log\log x)^{-3/8}},\\
	P(x)&=\Omega\parentheses*{x^{1/4}(\log x)^{1/4}(\log\log x)^{(3/4)(2^{1/3}-1)}(\log\log\log x)^{-3/8}}.
\end{align*}
\end{cor}
\begin{rem}
The improvement in comparison to \eqref{eq:circle_and_divisor} might at first sight seem minimal and it is also not possible to specify whether $\Omega$ is $\Omega_+$ or $\Omega_-$ because Theorem \ref{thm:main_threorem} produces a positive lower bound for the absolute value of ${\Re(e^{i\beta}F(x))}$.
On the other hand, it has been conjectured in \cite{Soundararajan_2003} based on a probabilistic argument, that the $\Omega$-result \eqref{eq:circle_and_divisor} for $\Delta(x)$ represents its true maximal order up to $(\log\log x)^{o(1)}$.
Further evidence towards the support of the  more precise conjecture 
\begin{align*}
	\begin{split}
		\max_{x\in[X,2X]}\Delta(x)&\asymp X^{1/4}(\log X)^{1/4}(\log\log X)^{(3/4)(2^{4/3}-1)}\\
		\max_{x\in[X,2X]}P(x)&\asymp X^{1/4}(\log X)^{1/4}(\log\log X)^{(3/4)(2^{1/3}-1)}.
	\end{split}
\end{align*}
were given recently by Lamzouri \cite[Conjecture 1.1, Conjecture 1.7]{Lamzouri2025} who employed a probabilistic model for $\Delta(x)$ and $P(x)$ to study the distribution of their large values.
In light of the above, increasing the power of the triple iteration of the logarithm (conjecturally up to $0$)  may be the only possible improvement one might hope to achieve for the lower bounds of these error terms.
\end{rem}
 Corollary \ref{cor:improved_circle_and_divisor} addresses only two of the many lattice point problems where Dirichlet's approximation theorem has been employed to produce $\Omega$-results. 
In Section \ref{applications} we will briefly see how to obtain it, since the proof of \cite[Theorem 1.1]{Soundararajan_2003} remains practically the same.
We provide several other examples where Theorem \ref{thm:main_threorem} and Theorem \ref{thm:main_subthreorem} can be used instead to obtain improved lower bounds but the list is definitely not complete.

In the special case when $\lambda_n=\log n$ and $f(n)$ is a multiplicative function, Theorem \ref{thm:main_threorem} and Theorem \ref{thm:main_subthreorem} do not represent the known $\Omega$-results for $F(x)$, because this is not the right polynomial to consider.
The theorems should instead be applied to a Dirichlet polynomial approximation of
$\log F(x)$ whose frequencies are given by $\lambda_{p^k}=k\log p$ for prime $p$ and integer $k\geq1$.
Here the following phenomenon occurs. 
If a resonator produces an extreme value $A>1$ when  applied to $|F(x)|$, then the same resonator  produces the extreme value $\log A$ when applied to $\log |F(x)|$ and vice versa.
In both cases the resonator is constructed in such way to interact with the same terms, namely  $f({p^k})\mathrm{e}(\lambda_{p^k}x)$. 
However, due to the multiplicativity of $f(n)$ and the additivity of $\lambda_n$, the remaining terms of $F(x)$ can also give a positive contribution to the lower bound which is then mutliplicatively reconstructed from the contribution coming from $f({p^k})$.
The resonance method should notably be applied to $\log|F(x)|$ with greater care because we have to take into account the zero-distribution of $F(x)$.
If we have sufficient knowledge about this distribution then the two approaches are in sync.
For instance, the resonator employed in the sequel reduces to the one constructed by Aistleitner in \cite{Aistleitner_2015} for proving lower bounds of $\zeta(s)$ to the right of the critical line. 
Applying the same resonator to $\log|\zeta(s)|$, in the sense of Theorem \ref{thm:main_threorem}, we could obtain the corresponding lower bound.
With additional information about $f(p^k)$, it is also possible to construct resonators with a richer structure than the one used in the present work, leading to better lower bounds.
Such constructions have been first realized by Soundararajan \cite{Soundararajan2008} and are culminating in the work of Bondarenko and Seip \cite{Bondarenko2017} and de la Bretèche and Tenenbaum \cite{Breteche2018} regarding lower bounds of $\zeta(s)$ on the critical line.

It is also possible to generalize the resonance method to trigonometric polynomials having complex Fourier coefficients.
In this case, however, the linear dependencies of the frequencies have also an effect on the length of the interval where $\Re(F(x))$ reaches a certain magnitude.
This is a natural by-product of the method and a common feature with the quantitative version of Kronecker's approximation theorem.
It turns out that the resonance method is one of the many techniques that have been devised to prove Kronecker's theorem and a result for general trigonometric polynomials in the form of Theorem \ref{thm:main_threorem} has been already obtained by Chen \cite{Chen2000}.
In the multiplicative setting we remark that this is the reason why we are able to obtain $\Omega$-results of similar strength when applying Soundararajan's or Chen's resonator to an $L$-function $L(s)$ or its logarithm $\log L(s)$, respectively, whose Dirichlet coefficients satisfy certain assumptions of number-theoretic interest.
In \cite{Aistleitner2017}  Aistleitner and Pa{\'n}kowski employ both approaches/resonators and make a detailed discussion about their strong and weak points.
They note that Chen's theorem produces inferior results only on the {\it critical line} of $L(s)$.
This is because Soundararajan's resonator is constructed from many more and much larger primes and it is tailor-made to capture this transition phase of the extremal behavior of $L(s)$ from the right of its critical line to the critical line.
If we were to apply this resonator to $\log |L(s)|$ or $\Im(\log L(s))$ then we could recover the expected lower bound as we discussed above.
Indeed this can be easily confirmed under the Riemann hypothesis for $L(s)$, namely that $L(s)$ has no zeros on the right of its critical line.
In point of fact, this is how Bondarenko and Seip \cite[Theorem 2]{Bondarenko2018} proved conditional lower bounds for the argument of $\zeta(s)$ on the critical line.

We will return to the discussion about Chen's theorem in the last section, where we will also prove an alternative version of it that matches Theorem \ref{thm:main_threorem}.
\begin{thm}\label{thm:Chen's_theorem}
{Let $F(x):=\sum_{n\leq N}f(n)\mathrm{e}(\lambda_n x)$ be a trigonometric polynomial with complex Fourier coefficients and $k\geq2$ be  fixed positive integer.}
	For positive integers $L, M$  with $M\leq N$ and a set $\mathcal{M}\subseteq\lbrace1,\dots,N\rbrace$ of cardinality $M$ let also $\delta=\delta(L,N,\mathcal{M})$ be defined as
	\[
	\min_{n\leq N}\set*{\abs*{j\lambda_n+\sum_{m\in\mathcal{M}}\ell_m\lambda_m}:j\in\set{0,1},\,\ell_m\in[-L,L]\cap\mathbb{Z},\,(j,(\ell_m)_{m\in\mathcal{M}})\neq(0,\dots,0)}.
	\]
	If $\delta>0$, then for any $T\geq1$
	\[
	\max_{x\in[T,2T]}\Re\parentheses*{F(x)}\geq\frac{L}{L+1}\sum_{m\in\mathcal{M}}|f(m)|-\frac{c(k)}{(\delta T)^k}\sum_{n\leq N}|f(n)|,
	\]
	where $c(k)>0$.
\end{thm}
%%%%%%%%%%%%%%%%%%%%%%%%%%%%%%%%%%%%%%%%%%%%%%%%%%%%%%%%%%
%%%%%%%%%%%%%%%%%%%%%%%%%%%%%%%%%%%%%%%%%%%%%%%%%%%%%%%%%%
%%%%%%%%%%%%%%%%%%%%%%%%%%%%%%%%%%%%%%%%%%%%%%%%%%%%%%%%%%
\section{Proofs of the main results}\label{Proofs}
Before giving the proofs of Theorem \ref{thm:main_threorem} and Theorem \ref{thm:main_subthreorem} we introduce two auxiliary functions. 
For any $u\in\mathbb{R}$ we denote by
\[
K(u):=\parentheses*{\frac{\sin(\pi u)}{\pi u}}^2\quad\text{ and }\quad\Phi(u):=e^{-\pi u^2}
\]
the Fejer kernel and the Gaussian function, respectively.
Their Fourier transforms satisfy
\[
\widehat{K}(u)=\int_{\mathbb{R}}K(x)\e{-ux}\mathrm{d}x=\max(0,1-|u|)\quad\text{ and }\quad\widehat{\Phi}(u)=\Phi(u)>0.
\] 
\begin{proof}[Proof of Theorem \ref{thm:main_threorem}]
	Starting as in the proof of \cite[Lemma 2.1]{Soundararajan_2003} with the convolution formula
	\begin{align*}
		\begin{split}
			\int_{\mathbb{R}}{\lambda_N}K\parentheses*{{\lambda_Nu}}\Re\parentheses*{e^{i\beta}F(x+u)}\e{-\lambda_N u}\mathrm{d}u
			&=\frac{e^{i\beta}}{2}\sum_{n\geq1}f(n)\e{\lambda_nx}\widehat{K}\parentheses*{\frac{\lambda_N-\lambda_n}{\lambda_N}}\\
			&=:\frac{e^{i\beta}}{2}F_1(x),
		\end{split}
	\end{align*}
	which is valid for any $x\in\mathbb{R}$,  we can show that
	\begin{align}\label{eq:convolution_inequality}
		\begin{split}
		\max_{|u|\leq Y/2}\abs*{\Re\parentheses*{e^{i\beta} F(x+u)}}
		&\geq\abs*{\int_{|u|\leq Y/2}{\lambda_N}K\parentheses*{{\lambda_Nu}}\Re\parentheses*{e^{i\beta}F(x+u)}\e{-\lambda_N u}\mathrm{d}u}\\
		&=\abs*{\frac{e^{i\beta}}{2}F_1(x)-\int_{|u|\geq Y/2}{\lambda_N}K\parentheses*{{\lambda_Nu}}\Re\parentheses*{e^{i\beta}F(x+u)}\e{-\lambda_N u}\mathrm{d}u}\\
		&\geq\frac{1}{2}{\Re\parentheses*{{F_{1}(x)}}}-\frac{4F(0)}{\pi^2Y\lambda_N}.
		\end{split}
	\end{align}
	Therefore, it suffices to obtain $\Omega$-results for $\Re\parentheses*{{F_{1}(x)}}$.
	To that end we define the sets
	\[
	\mathcal{A}:=\set*{\sum_{m\in\mathcal{M}}\varepsilon_m\lambda_m:\varepsilon_m=0\text{ or }1},\quad \mathcal{B}_j:=\mathcal{A}\cap\left[\frac{(j-1)\log T}{T},\frac{j\log T}{T}\right),\quad j\geq1,
	\]
	and
	\[
	\bigsqcup_{\substack{\mathcal{B}_{j}\neq\emptyset}}\set{\min \mathcal{B}_{j}}=\bigsqcup_{\substack{q\leq Q}}\set{\min \mathcal{B}_{j_q}}=:\set*{d_q:q\leq Q}.
	\]
	Notice that $Q\leq\#\mathcal{A}= 2^M$ and $(j_q)_{q\leq Q}$ is a strictly increasing sequence of positive integers.
	We then set
	\[
	R(x):=\sum_{q\leq Q}\e{d_qx},\quad x\in\mathbb{R},
	\]
	and our goal is to compute the ratio of the quantities
	\[
	J_1=\int_\mathbb{R}F_1(x)|R(x)|^2 \Phi\parentheses*{\frac{x\log T}{T}}\mathrm{d}x\quad\text{ and }\quad J_2:=\int_{\mathbb{R}}|R(x)|^2\Phi\parentheses*{\frac{x\log T}{T}}\mathrm{d}x.
	\]
	
	Starting with $J_2$,  squaring out and interchanging summations and integration yield that
	\[
	J_2=\frac{QT\widehat{\Phi}(0)}{\log T}+\frac{T}{\log T}\mathop{\sum\sum}_{q\neq r\leq Q}\widehat{\Phi}\parentheses*{\frac{T(d_q-d_r)}{\log T}}.
	\]
	Since $\widehat{\Phi}(x)=\Phi(x)>0$ and $\frac{T(d_q-d_r)}{\log T}>j_q-j_r-1\geq0$ for $q> r$, we deduce that
	\begin{align}\label{eq:resonator_bounds}
		\frac{QT}{\log T}\leq J_2
		\leq \frac{QT}{\log T}+\frac{T}{\log T}\mathop{\sum\sum}_{q\neq r\leq {Q}}\Phi(|j_q-j_r|-1)\leq\frac{CQT}{\log T},
			\end{align}
	where the last relation, with $C>0$ some absolute constant, follows from the exponential decay of $\Phi(x)$.
	
	We proceed on the estimation of $J_1$.
	{Since $F_1(x)$ and $R(x)$ are absolutely convergent}, we can square out and interchange summations and integration to obtain that
	\begin{align*}
		J_1
		&=\frac{T}{\log T}\sum_{q,r\leq Q}\sum_{n\geq1}f(n)\widehat{K}\parentheses*{\frac{\lambda_N-\lambda_n}{\lambda_N}}\widehat{\Phi}\parentheses*{\frac{T(\lambda_{n}+d_q-d_r)}{\log T}}.
	\end{align*}
	Observe that all terms of the above expression are positive. 
	Therefore $J_1\in\mathbb{R}_+$ and
	\begin{align}\label{eq:bounds_resonated_polynomial}
		J_1\geq \frac{T}{\log T}\sum_{m\in\mathcal{M}}f(m)\widehat{K}\parentheses*{\frac{\lambda_N-\lambda_m}{\lambda_N}}\sum_{q,r\leq Q}\Phi\parentheses*{\frac{ T|\lambda_m+d_q-d_r|}{\log T}}.
	\end{align}
	
	Let now $m\in\mathcal{M}$ be fixed and set \[
	\mathcal{Q}_m:=\set*{q\leq{Q}:d_q+\lambda_m\in\mathcal{A}}.
	\]
	Then either $\#\mathcal{Q}_m\geq Q/2$ or $\#\mathcal{Q}_m<Q/2$.
	If the first is true, then we argue as follows.
	All of the numbers $d_q+\lambda_m$, $q\in\mathcal{Q}_m$, are elements of $\mathcal{A}$ and so we can find $r_q\leq Q$ with $d_q+\lambda_m\in \mathcal{B}_{j_{r_q}}$.
	Hence,  $T|\lambda_m+d_q-d_{r_q}|\leq\log T$ by construction and, consequently,
	\[
	\sum_{q\leq Q}\sum_{r\leq Q}\Phi\parentheses*{\frac{T|\lambda_m+d_q-d_r|}{\log T}}\geq \sum_{q\in\mathcal{Q}_m}\Phi\parentheses*{\frac{T|\lambda_m+d_q-d_{r_q}|}{\log T}}\geq \sum_{q\in\mathcal{Q}_m}\Phi(1)\geq \frac{\Phi(1)Q}{2}.
	\]
	The second case $\#\mathcal{Q}_m<Q/2$ is treated similarly by considering the elements $d_r-\lambda_m\in\mathcal{A}$, $r\notin\mathcal{Q}_m$.
	Thus, for any $m\in\mathcal{M}$, the above double sum is greater than $\Phi(1){Q}/{2}$ and in view of \eqref{eq:resonator_bounds} and \eqref{eq:bounds_resonated_polynomial}
	we deduce that
	\begin{align}\label{eq:ratio}
		\frac{J_1}{J_2}\geq\frac{\Phi(1)}{2C}\sum_{m\in\mathcal{M}}f(m)\widehat{K}\parentheses*{\frac{\lambda_N-\lambda_m}{\lambda_N}}\geq\frac{\Phi(1)}{4C}\sum_{\substack{m\in\mathcal{M}}}f(m).
	\end{align}
	{The last inequality followed from $\widehat{K}\parentheses*{\frac{\lambda_N-\lambda_m}{\lambda_N}}\geq1-\frac{|\lambda_N-\lambda_m|}{\lambda_N}\geq\frac{1}{2}$, valid by assumption for any $m\in\mathcal{M}$.}
	It remains to {relate} $J_1/J_2$ with the maximum of $\Re\left(F_1(x)\right)$ in a certain interval of $x$. 
	The trivial estimates $|F_1(x)|\leq F_1(0)$ and $|R(x)|\leq R(0)=Q\leq 2^M$  imply that
	\begin{align*}
		\braces*{\int_{|x|\geq T}+\int_{|x|\leq Y}}F_1(x)|R(x)|^2\Phi\parentheses*{\frac{x\log T}{T}}\mathrm{d}x\ll{F_1(0)Q^2Y}\ll\frac{QT}{\log T}\frac{F_1(0)2^MY\log T}{T}.
	\end{align*}
	Thus, we obtain that
	\begin{align*}
		\int_{|x|\in[Y, T]}F_1(x)|R(x)|^2\Phi\parentheses*{\frac{x\log T}{T}}\mathrm{d}x
		=J_1+O\parentheses*{\frac{QT}{\log T}\frac{F_1(0)2^MY\log T}{T}}.
	\end{align*}
	Observe that the real part of the integrand above is an even function.
	Therefore, in view of relation \eqref{eq:resonator_bounds} we have that
	\begin{align*}
		\begin{split}
			\max_{x\in[Y, T]}\Re\parentheses*{F_1(x)}
			&\geq \frac{1}{J_2}{	\Re\brackets*{\int_{|x|\in[Y, T]}F_1(x)|R(x)|^2\Phi\parentheses*{\frac{x\log T}{T}}\mathrm{d}x}}\\
			&=\frac{J_1}{J_2}+O\parentheses*{\frac{2^MY\log T}{T}\sum_{\lambda_n\leq2\lambda_{N}}f(n)}.
		\end{split}
	\end{align*}
	The theorem follows now from the above and relations \eqref{eq:ratio} and \eqref{eq:convolution_inequality}.
\end{proof}
\begin{proof}[Proof of Theorem \ref{thm:main_subthreorem}]
	Repeating the proof of Theorem \ref{thm:main_threorem} without employing the convolution formula and treating $F(x)$ instead of $F_1(x)$ leads to the following inequality 
	\begin{align}\label{eq:proof_without_rotation}
		\int_\mathbb{R}\Re\parentheses*{F(x)}|R(x)|^2 \Phi\parentheses*{\frac{x\log T}{T}}\mathrm{d}x\geq\frac{\Phi(1)}{4C} \sum_{m\in\mathcal{M}}f(m)\int_\mathbb{R}|R(x)|^2 \Phi\parentheses*{\frac{x\log T}{T}}\mathrm{d}x,
	\end{align}
	whence
	\[
	\max_{x\in[Y,T]}\Re\parentheses*{{F}(x)}\geq\frac{1}{J_2}\int_{|x|\in[Y,T]}\Re\parentheses*{{F}(x)}|R(x)|^2\Phi\parentheses*{\frac{x\log T}{T}}\mathrm{d}x\geq\Sigma.
	\]
	If we set now
	\[
	{\mathcal{S}}:=\set*{|x|\in[Y,T]:{\Re\parentheses*{{F}(x)}}\geq V}
	\]
	then
	\[
	\Sigma\leq\frac{1}{J_2}\int_{{\mathcal{S}}}\Re\parentheses*{{F}(x)}|R(x)|^2\Phi\parentheses*{\frac{x\log T}{T}}\mathrm{d}x+V
	\]
	and since $|R(x)|^2\leq Q^2\leq 4^M$ and $J_2\geq\frac{QT}{\log T}$, we deduce that
	\[
	\meas({\mathcal{S}})\geq\frac{(\Sigma-V) T}{2^M
		\log T\max_{x\in[Y,T]}|F(x)|}.
	\]
\end{proof}
%%%%%%%%%%%%%%%%%%%%%%%%%%%%%%%%%%%%%%%%%%%%%%%%%%%%%%%%%%
%%%%%%%%%%%%%%%%%%%%%%%%%%%%%%%%%%%%%%%%%%%%%%%%%%%%%%%%%%
%%%%%%%%%%%%%%%%%%%%%%%%%%%%%%%%%%%%%%%%%%%%%%%%%%%%%%%%%%
	\section{Applications}\label{applications}
	Here and in the sequel we adopt the notation $\log_{j+1}:=\log\log_j$, $j\geq1$, with $\log_1:=\log$.
	Moreover, for the discussion in this section $c>0$ will denote a sufficiently small constant.
	
	We begin with the error term $\Delta(x)$ on Dirichlet's divisor problem.
	By an elementary argument Dirichlet was able to show that $\Delta(x)\ll\sqrt{x}$.
	The first improvement on the order of magnitude of $\Delta(x)\ll_\epsilon x^{\theta+\epsilon}$ using methods from analytic number theory has been realized by Vorono\"i, who first developed and then employed a variant of the following formula (see \cite[(12.4.4)]{Titchmarsh1986} or \cite[(3.1)]{Soundararajan_2003}) 
	\begin{align*}
		\frac{\pi\sqrt{2}\Delta\parentheses*{x^2}}{x^{1/2}}+o(1)=\Re\parentheses*{e^{-\pi i/4}\sum_{n\leq X^3}\frac{d(n)}{n^{3/4}}\mathrm{e}\parentheses*{2\sqrt{n}x}}=:\Re\parentheses*{e^{i\beta}F(x)},\quad x\in\brackets*{\sqrt{X},X^{3/2}}.
	\end{align*}
	More advanced techniques have been utilized over the past century with the best result to date being due to Huxley \cite{Huxley2003} who showed that $\theta=\frac{131}{416}=0.31490..$ is admissible and it is conjectured that it can be reduced to $1/4$.
	In the opposite direction, the above formula has been also employed to produce $\Omega$-results.
	\begin{proof}[Sketch proof of Corollary \ref{cor:improved_circle_and_divisor}]
	Soundararajan \cite[Theorem 1.1]{Soundararajan_2003} applied inequality \eqref{eq:motivation} to the right-hand side of Vorono\"i's formula with
	\[
	\mathcal{M}:=\set*{m\in[N/4,9N/4]:\omega(m)=\floor{\lambda\log_2N}},
	\] where $\omega(n)$ denotes the number of prime divisors of $n\in\mathbb{N}$ and $\lambda$ is optimally chosen in the end.
By a classic result (see for example \cite[Chapter II.6, Theorem 6.4]{Tenenbaum2015}) he derived that
	\begin{align}\label{eq:Sathe_Selberg}
	M=\#\mathcal{M}\asymp\frac{N(\log N)^{\lambda-1-\lambda\log\lambda}}{\sqrt{\log_2N}}
	\end{align}
	and, in view of $d(m)\geq 2^{\omega(m)}$, he concluded that
	\[
\max_{x\in[X/2,(6L)^{M+1}X]}\abs*{\Re\parentheses*{e^{i\beta}F(x)}}\gg\frac{N^{1/4}(\log N)^{\lambda-1-\lambda\log\lambda+\lambda\log2}}{\sqrt{\log_2 N}}+O\parentheses*{\frac{\log X}{X^{1/4}}+\frac{N^{1/4}\log N}{L}}.
	\]
	For the first term on the right-hand side above to dominate, we require that 
	\[
2-\lambda+\lambda\log\lambda-\lambda\log2\geq0\quad\text{ and	}\quad L\gg(\log N)^{2-\lambda+\lambda\log\lambda-\lambda\log2}\sqrt{\log_2N}.
	\]
At the same time we want to be able to return to $\Delta\parentheses*{x^2}$ and so $M\log L\leq c\log X$ should also hold in order for $[X/2,(6L)^{M+1}X]\subseteq[\sqrt{X},X^{3/2}]$.
Optimizing this argument yields that $N=c\log X(\log_2X)^{1-\lambda+\lambda\log\lambda}(\log_3 X)^{-1/2}$ and $\lambda=2^{4/3}$.

If we instead apply Theorem \ref{thm:main_threorem} with the same set of integers $\mathcal{M}$, $Y={X}$ and $T=2^MX^{5/4}$, we then get that
	\begin{align*}
\max_{x\in\brackets*{{X},2^{M+1}X^{5/4}}}\abs*{\Re\parentheses*{e^{i\beta}F(x)}}
&\gg\frac{N^{1/4}(\log N)^{\lambda-1-\lambda\log\lambda+\lambda\log2}}{\sqrt{\log_2 N}}\\
&\quad+O\parentheses*{\frac{\log X}{X^{1/4}}+\frac{(M+\log X)N^{1/4}\log N}{X^{1/4}}}.
\end{align*}
Now we only require that $M\leq c\log X$ which in view of relation \eqref{eq:Sathe_Selberg} yields that $N=c\log X(\log_2X)^{1-\lambda+\lambda\log\lambda}(\log_3 X)^{+1/2}$ and $\lambda=2^{4/3}$ are the optimal choices.
This proves the first part of Corollary \ref{cor:improved_circle_and_divisor}, while the second part follows with minimal modifications.
\end{proof}
The above reasoning and improvements apply also to the error term for the $k$-divisor problem (see \cite[Theorem 1.2]{Soundararajan_2003}) or its generalization to number fields (see \cite[Theorem 1]{Girstmair_2005}) and the error term of the mean square of $\zeta(s)$ (see \cite[Theorem 1.1]{Lau_2005}) to mention a few.
	\begin{cor}\label{cor:divisork} 
		Let $k\geq2$ be fixed and
		\[
		\Delta_k(x):=\sum_{n\leq x}d_k(n)-\mathrm{Res}\brackets*{x^{s-1}\zeta^k(s)s^{-1}}_{s=1},\quad x\geq1,
		\]
		be the error term in the $k$-divisor problem.
		Then 
		\[
		\Delta_k(x)=\Omega_*\parentheses*{(x\log x)^{\frac{k-1}{2k}}(\log_2x)^{(k^{2k/(k+1)}-1)\frac{k+1}{2k}}(\log_3 x)^{-1/2+\frac{k-1}{4k}}},
		\]
		where $\Omega_*=\Omega_+$ or $\Omega_-$ if $k\equiv3$ or $7$ $(\bmod\,8)$, respectively, and $\Omega_*=\Omega$ otherwise.
	\end{cor}
	\begin{proof}[Sketch proof of Corollary \ref{cor:divisork}]
Following the construction in the proof of \cite[Theorem 1.2]{Soundararajan_2003}, it suffices to show that
\begin{align}\label{eq:convolution_divisork}
\max_{x\in[X/2,X^{3/2}]}\abs*{\Re\parentheses*{e^{(k-3)\pi i/4}F(x)}}\gg(\log X)^{\frac{k-1}{2k}}(\log_2X)^{(k^{2k/(k+1)}-1)\frac{k+1}{2k}}(\log_3 X)^{-1/2+\frac{k-1}{4k}},
\end{align}
where
\[
F(x):=\sum_{n\geq1}\frac{d_k(n)}{n^{(k+1)/(2k)}}\exp\parentheses*{-\pi^2\parentheses*{\frac{n}{N}}^{2/k}}\mathrm{e}\parentheses*{kn^{1/k}x}.
\]
We let  $\mathcal{M}$ be the set of positive integers in $[2^{-k}N,(3/2)^kN]$ that have exactly $\floor{\lambda\log_2N}$ distinct prime factors.
We have already mentioned before that its cardinality is $M\asymp N(\log N)^{\lambda-1-\lambda\log\lambda}(\log_2N)^{-1/2}$.
 We employ now Theorem \ref{thm:main_threorem} with $Y=X$, $T=2^MX^{5/4}$ and this  set $\mathcal{M}$ to deduce, in view of $d_k(m)\geq k^{\omega(m)}$, that
 \begin{align*}
 \max_{x\in[X/2,2^{M+1}X^{5/4}]}\abs*{\Re\parentheses*{e^{(k-3)\pi i/4}F(x)}}
 &\gg  \frac{N^{(k-1)/(2k)}(\log N)^{\lambda-1-\lambda\log\lambda+\lambda\log k}}{\sqrt{\log_2N}}\\
 &\quad+O_\epsilon\parentheses*{\frac{M+\log X}{X^{1/4}}N^{(k-1)/(2k)+\epsilon}}
 \end{align*}
We wish to maximize the first term on the right-hand side of the above relation with the restriction $M\leq c\log X$ in order for $[X/2,2^{M+1}X^{5/4}]\subseteq[X/2,X^{3/2}]$.
This
implies that $N=c\log X(\log_2 X)^{1+\lambda\log \lambda-\lambda}(\log_3 X)^{+1/2}$ and $\lambda=k^{2k/(k+1)}$ are the optimal choices, in which case we obtain \eqref{eq:convolution_divisork}.

When $k\equiv3$ or $7$ $(\bmod\,8)$, we repeat the above procedure where we employ Theorem \ref{thm:main_subthreorem} instead of Theorem \ref{thm:main_threorem} to obtain a positive lower bound for $\max_{x\in[X/2,X^{3/2}]}\Re\parentheses*{F(x)}$.
This produces the desired $\Omega_*$-results.
Moreover, by utilizing upper bounds for the relevant error terms in these two cases (see \cite[Chapter XII]{Titchmarsh1986}), we could also derive lower bounds for the measure of those $x\in[X,X^{3/2}]$ for which $x^{\frac{1-k}{2k}}\Delta_k(x)$ (or $-x^{\frac{1-k}{2k}}\Delta_k(x)$, depending on $k$) is as large as $V\ll(\log X)^{\frac{k-1}{2k}+o(1)}$.
\end{proof}

	\begin{cor}
		Let 
		\[
		E(t):=\int_{0}^t\abs*{\zeta\parentheses*{\frac{1}{2}+iu}}^2\mathrm{d}u-t\log\frac{t}{2\pi}-(2\gamma-1)t,\quad t\geq1,
		\]
		be the error term for the mean square of $\zeta(s)$. 
		Then 
		\[
		E(t)=\Omega\parentheses*{(t\log t)^{1/4}(\log_2 t)^{(3/4)(2^{4/3}-1)}(\log_3 t)^{-3/8}}.
		\]
	\end{cor}
The proof of the corollary follows  {\it mutatis mutandis} that of \cite[Theorem 1.1]{Lau_2005} with the only difference being in Lemma 2.2 therein. 
In our case the admissible choices are $\tau\asymp (\log X)^{1/2}(\log_2 X)^{(1+\lambda\log \lambda-\lambda)/2}(\log_3 X)^{1/4}$ and $\lambda=2^{4/3}$  which yield the desired improvement.
%%%%%%%%%%%%%%%%%%%%%%%%%%%%%%%%%%%%%%%%%%%%%%%%%%%%%%%%%%
%%%%%%%%%%%%%%%%%%%%%%%%%%%%%%%%%%%%%%%%%%%%%%%%%%%%%%%%%%
%%%%%%%%%%%%%%%%%%%%%%%%%%%%%%%%%%%%%%%%%%%%%%%%%%%%%%%%%%
\section{The link to Kronecker's theorem}
We conclude with a discussion on the resonance method for trigonometric polynomials
\[
F(x)=\sum_{n\leq N}f(n)\mathrm{e}(\lambda_n x),\quad x\in\mathbb{R},
\]
having complex coefficients $f(n)$.
The origins of it can be traced back to Bohr and Jessen's proof \cite{Bohr1932} of Kronecker's approximation theorem.
\begin{thm*}[Kronecker's theorem]\label{Similt}
	Let $\lambda_1,\dots,\lambda_N$ be real numbers  linearly independent over $\mathbb{Q}$.
	For any real numbers $\alpha_1,\dots,\alpha_N$ and $\varepsilon>0$, there is $x_0\in\mathbb{R}$, such that\footnote{Here $\norm{x}$ denotes the distance of $x$ from the nearest integer.}
	\[
	\max_{n\leq N}\norm{x_0\lambda_n-\alpha_n}<\varepsilon.
	\]
\end{thm*}
Bohr and Jessen considered the polynomial
\[
F(x)=1+\sum_{n\leq N}\mathrm{e}(-\alpha_n)\mathrm{e}(\lambda_{n}x),\quad x\in\mathbb{R},
\]
and showed that $\sup_{x\in\mathbb{R}}\Re(F(x))= N+1$, which implies the theorem.
It is clear that the supremum is less than or equal to $N+1$ but to obtain the equality they employed the positive trigonometric polynomial
\[
W_L(x):=K_L(x\lambda_1-\alpha_1)\cdots K_L(x\lambda_N-\alpha_N),
\]
where
\[
 K_L(x):=\sum_{|\ell|\leq L}\parentheses*{1-\frac{|\ell|}{L}}\mathrm{e}(\ell x)={\frac{1}{L}\parentheses*{\frac{\sin(L\pi x)}{\sin(\pi x)}}^2}.
\]
Multiplying out and using the linear independency of the frequencies it can be shown that
\[
F(x)W_L(x)=1+\frac{L-1}{L}N+S(x),
\]
 where $S(x)$ is a trigonometric polynomial whose frequencies are all different from $0$.
 They concluded their argument by proving that
 \begin{align}\label{eq:Bohr_Jessen}
 	\begin{split}
 \sup_{x\in\mathbb{R}}\Re(F(x))
 &\geq\parentheses*{\lim_{T\to\infty}\frac{1}{2T}\int_{-T}^TW_L(x)\mathrm{d}x}^{-1}\Re\parentheses*{\lim_{T\to\infty}\frac{1}{2T}\int_{-T}^TF(x)W_L(x)\mathrm{d}x}\\
 &=1\cdot\parentheses*{1+\frac{L-1}{L}N}
 \end{split}
 \end{align}
and then taking $L$ to infinity.
Note that $W_{L}(x)$ acts as a resonating polynomial for $F(x)$ and this can also be seen through its expansion
\[
W_L(x)=\sum_{\mu\in\mathcal{A}_L}c(\mu)\mathrm{e}(\mu x),\quad\mathcal{A}_L:=\set*{\sum_{n\leq N}\ell_n\lambda_n:\ell_n\in[-L,L]\cap\mathbb{Z}},\quad |c(\mu)|\leq 1.
\]
{Since the frequencies  $\mu\in\mathcal{A}_L$ of $W_L(x)$ are $\mathbb{Z}$-linear combinations of $\lambda_n$, each of them can be represented as $\mu=\mu_1-\mu_2$ for some $\mu_1,\mu_2\in\mathcal{A}_L$ that are $\mathbb{N}_{\geq0}$-linear combinations of $\lambda_n$.
Based on this observation one can replace $c(\mu)$ with suitable coefficients $\tilde{c}(\mu)$ to generate $|R_L(x)|^2$, where $R_L(x)$ is a trigonometric polynomial whose frequencies are $\mathbb{N}_{\geq0}$-linear combination of $\lambda_n$, in such a way that it does the exact same job as $W_L(x)$ in \eqref{eq:Bohr_Jessen}.
The limits will be different and the ratio will be $(N+1)L/(L+1)$, but when taking $L$ to infinity the same result follows.
This informal argument will  be made precise in the proof of Theorem \ref{thm:Chen's_theorem} below.}

There is a rich literature on quantitative versions of Kronecker's theorem, that is to find explicitly a $T$ so that the theorem holds for some $x_0\leq T$, and we refer to the work of Gonek and Montgomery \cite{Gonek2016} for a detailed exposition and historical review.
One of the techniques is naturally based on not taking $T$ to infinity in \eqref{eq:Bohr_Jessen} and estimating instead the least $T_0$ such that
the occurring error terms when divided by $T_0$ would be $\ll 1+N(L-1)/L$.
This is of course the resonance method for producing large values of $F(x)$ in the interval $[0,T_0]$ and its optimal application up to constants has been proved by Chen \cite{Chen2000} who used slightly different positive  trigonometric polynomials to construct $W_L(x)$.
%We state a variation of the theorem by following the proof in \cite[Theorem 5.1]{Gonek2016}.

\begin{thm*}[Chen's theorem]\label{Chensthm}
For positive integers $L$, $N$  let 
\[
\delta=\delta(L,N):=\min\set*{\abs*{\sum_{n\leq N}\ell_n\lambda_n}:\ell_n\in[-L,L]\cap\mathbb{Z},\,(\ell_1,\dots,\ell_N)\neq(0,\dots,0)}.
\]
If $\delta>0$, then for any $T\geq1$
\[
\max_{x\in[T,2T]}\Re\parentheses*{F(x)}\geq\parentheses*{1-\frac{\pi^4}{8(L+1)^2}-\frac{\pi L^N}{2T\delta}}\sum_{n\leq N}|f(n)|.
\]
\end{thm*}
In most cases the theorem is given in an equivalent form which is based on the inequalities $1-2\pi^2\norm{x}^2\leq\cos(2\pi x)\leq1-8\norm{x}^2$, $x\in\mathbb{R}$.
Chen proves a more general result where linear dependencies between the frequencies are allowed, but it would not give further insight on the resonance method and so it is omitted.
It is very likely that we can modify its proof to give a lower bound with respect to a given set of frequencies $\lambda_m$, $m\in\mathcal{M}\subseteq\set{1,\dots,N}$.
For the sake of keeping the exposition short and consistent with the method presented in the previous section, it has been preferred instead to show a slightly different result.

\begin{proof}[Proof of Theorem \ref{thm:Chen's_theorem}]
	Let $\Phi\in C^\infty(\mathbb{R})$ be a smooth function supported in $[1,2]$ such that $0\leq\Phi(x)\leq1$ and $\widehat{\Phi}(0)>0$.
	Partial integration yields for any integer $k\geq1$ that $\widehat{\Phi}(x)\ll_k(1+|x|)^{-k}$, $x\in\mathbb{R}$.
	Let also $g:\mathcal{A}_\mathcal{M}\to\mathbb{T}$ be a function defined on the set
	\[
	\mathcal{A}_\mathcal{M}:=\set*{\sum_{m\in\mathcal{M}}\ell_m\lambda_m:\ell_m\in[0,L]\cap\mathbb{Z}}
	\]
	in a completely additive way. 
	In particular, if $f(n)=|f(n)|\mathrm{e}(\alpha_n)$ then we let
	\[
	g(\lambda_m):=\mathrm{e}(\alpha_m)\quad\text{ and }\quad g\parentheses*{\sum_{m\in\mathcal{M}}\ell_m\lambda_m}:=\prod_{m\in\mathcal{M}}g(\lambda_m)^{\ell_m}.
	\]
	For later reference we also let 
	\[
	\mathcal{A}_n:=\set*{j\lambda_n+\mu:j\in\set*{0,1},\,\mu\in\mathcal{A}_\mathcal{M}},\quad n\leq N.
	\]
Lastly, we define the resonator
\[
R(x):=\sum_{\mu\in\mathcal{A}_\mathcal{M}}g(\mu)\mathrm{e}(\mu x),\quad x\in\mathbb{R},
\]
and we note that
	\begin{align}\label{complex_resonance}
	\max_{t\in[T,2T]}\Re(F(x))\geq\parentheses*{\int_{\mathbb{R}}\abs*{R(x)}^2\Phi\parentheses*{\frac{x}{T}}\mathrm{d}x}^{-1}\Re\parentheses*{\int_{\mathbb{R}}F(x)\abs*{R(x)}^2\Phi\parentheses*{\frac{x}{T}}\mathrm{d}x}=:\frac{\Re(I_1)}{I_2}.
	\end{align}
	
	Opening the square and interchanging summations and integration in $I_2$ yields that
	\[
	I_2=T\sum_{\mu,\kappa\in\mathcal{A}_\mathcal{M}}g(\mu)\overline{g(\kappa)}\widehat{\Phi}(T(\mu-\kappa)).
	\]
	Since $\mu_1-\mu_2\geq\delta$ for any pair of consecutive elements $\mu_1>\mu_2$ from $\mathcal{A}_\mathcal{M}$, we deduce that
	\begin{align*}
	I_2&
	=T(L+1)^M\widehat{\Phi}(0)+O_k\parentheses*{T\mathop{\sum\sum}_{\mu\neq\kappa\in\mathcal{A}_\mathcal{M}}(\delta T)^{-k}\parentheses*{\frac{\delta}{|\mu-\kappa|}}^{k}}\\
	&=T(L+1)^M\widehat{\Phi}(0)\parentheses*{1+O_k\parentheses*{(\delta T)^{-k}}}.
	\end{align*}
Repeating the above procedure for $I_1$ gives
\begin{align*}
		I_1
		&=T\sum_{n\leq N}f(n)\sum_{\mu,\kappa\in\mathcal{A}_\mathcal{M}}g(\mu)\overline{g(\kappa)}\widehat{\Phi}(T(\mu+\lambda_n-\kappa))\\
&=T\widehat{\Phi}(0)\sum_{m\in\mathcal{M}}|f(m)|\sum_{\substack{\mu,\kappa\in\mathcal{A}_\mathcal{M}\\\mu+\lambda_m=\kappa}}1+O_k\parentheses*{T\sum_{n\leq N}|f(n)|\mathop{\sum\sum}_{\mu\neq\kappa\in\mathcal{A}_{n}}(\delta T)^{-k}\parentheses*{\frac{\delta}{|\mu-\kappa|}}^{k}}\\
	&=TL(L+1)^{M-1}\widehat{\Phi}(0)\parentheses*{\sum_{m\in\mathcal{M}}|f(m)|+O_k\parentheses*{(\delta T)^{-k}\sum_{n\leq N}|f(n)|}}.
	\end{align*}
The theorem follows now by taking the ratio $I_1/I_2$ in \eqref{complex_resonance}.
\end{proof}

{\paragraph{\bf Acknowledgments} I am grateful to the reviewer for thoroughly reading my work and whose comments helped me improve and clarify several parts of the exposition.
This material is based upon work supported by the Swedish Research Council under grant no. 2021-06594 while I was in residence at Institut Mittag-Leffler in Djursholm, Sweden during the year of 2024.
I would particularly like to thank Jake Chinis, whom I met there, for introducing me to this type of questions.}
		\bibliographystyle{plainnat}
		\bibliography{Omega}

\begin{thebibliography}{21}
\providecommand{\natexlab}[1]{#1}
\providecommand{\url}[1]{\texttt{#1}}
\expandafter\ifx\csname urlstyle\endcsname\relax
  \providecommand{\doi}[1]{doi: #1}\else
  \providecommand{\doi}{doi: \begingroup \urlstyle{rm}\Url}\fi

\bibitem[Aistleitner(2015)]{Aistleitner_2015}
C.~Aistleitner.
\newblock Lower bounds for the maximum of the {R}iemann zeta function along
  vertical lines.
\newblock \emph{Math. Ann.}, {\bf365}\,\penalty0 (1–2):\penalty0 \,473--496,
  2015.
\newblock \doi{10.1007/s00208-015-1290-0}.

\bibitem[Aistleitner and Pańkowski(2017)]{Aistleitner2017}
C.~Aistleitner and Ł. Pańkowski.
\newblock Large values of ${L}$-functions from the {S}elberg class.
\newblock \emph{J. Math. Anal. Appl.}, {\bf 446}\,\penalty0 (1):\penalty0
  \,345--364, 2017.
\newblock \doi{10.1016/j.jmaa.2016.08.044}.

\bibitem[Aistleitner et~al.(2019)Aistleitner, Mahatab, and
  Munsch]{Aistleitner2017a}
C.~Aistleitner, K.~Mahatab, and M.~Munsch.
\newblock Extreme values of the {R}iemann zeta function on the 1-line.
\newblock \emph{Int. Math. Res. Notices}, {\bf2019}\,\penalty0 (22):\penalty0
  \,6924--6932, 2019.
\newblock \doi{10.1093/imrn/rnx331}.

\bibitem[Bohr and Jessen(1932)]{Bohr1932}
H.~Bohr and B.~Jessen.
\newblock One more proof of {K}ronecker’s theorem.
\newblock \emph{J. London Math. Soc.}, {\bf 7}\,\penalty0 (4):\penalty0
  \,274--275, 1932.
\newblock \doi{10.1112/jlms/s1-7.4.274}.

\bibitem[Bondarenko and Seip(2017)]{Bondarenko2017}
A.~Bondarenko and K.~Seip.
\newblock Large greatest common divisor sums and extreme values of the
  {R}iemann zeta function.
\newblock \emph{Duke Math. J.}, {\bf 166}\,\penalty0 (9):\penalty0
  \,1685--1701, 2017.
\newblock \doi{10.1215/00127094-0000005x}.

\bibitem[Bondarenko and Seip(2018)]{Bondarenko2018}
A.~Bondarenko and K.~Seip.
\newblock Extreme values of the {R}iemann zeta function and its argument.
\newblock \emph{Math. Ann.}, {\bf 372}\,\penalty0 (3–4):\penalty0 \,99--1015,
  2018.
\newblock \doi{10.1007/s00208-018-1663-2}.

\bibitem[Chen(2000)]{Chen2000}
Y.-G. Chen.
\newblock The best quantitative {K}ronecker’s theorem.
\newblock \emph{J. London Math. Soc.}, {\bf 61}\,\penalty0 (3):\penalty0
  \,691--705, 2000.
\newblock \doi{10.1112/s0024610700008772}.

\bibitem[Corr\'adi and K\'atai(1967)]{Corradi1967}
K.~Corr\'adi and I.~K\'atai.
\newblock A comment on {K}. {S}. {G}angadharan's paper entitled ``{T}wo
  classical lattice point problems''.
\newblock \emph{Magyar Tud. Akad. Mat. Fiz. Oszt. K\"ozl.}, {\bf
  17}\,:\penalty0 \,89--97, 1967.

\bibitem[de~la Bretèche and Tenenbaum(2018)]{Breteche2018}
R.~de~la Bretèche and G.~Tenenbaum.
\newblock Sommes de {G}ál et applications.
\newblock \emph{Proc. London Math. Soc. (3)}, {\bf 119}\,\penalty0
  (1):\penalty0 \,104--134, 2018.
\newblock \doi{10.1112/plms.12224}.

\bibitem[Girstmair et~al.(2005)Girstmair, Kühleitner, Müller, and
  Nowak]{Girstmair_2005}
K.~Girstmair, M.~Kühleitner, W.~Müller, and W.~G. Nowak.
\newblock The {P}iltz divisor problem in number fields: {A}n improved lower
  bound by {S}oundararajan’s method.
\newblock \emph{Acta Arith.}, {\bf117}\,\penalty0 (2):\penalty0 \,187--206,
  2005.
\newblock \doi{10.4064/aa117-2-6}.

\bibitem[Gonek and Montgomery(2016)]{Gonek2016}
S.~M. Gonek and H.~L. Montgomery.
\newblock Kronecker’s approximation theorem.
\newblock \emph{Indag. Math. (N.S.)}, {\bf 27}\,\penalty0 (2):\penalty0
  \,506--523, 2016.
\newblock \doi{10.1016/j.indag.2016.02.002}.

\bibitem[Hafner(1981)]{Hafner1981}
J.~L. Hafner.
\newblock New omega theorems for two classical lattice point problems.
\newblock \emph{Invent. Math.}, {\bf 63}\,\penalty0 (2):\penalty0 \,181--186,
  1981.
\newblock \doi{10.1007/bf01393875}.

\bibitem[Hardy(1916)]{Hardy1917}
G.~H. Hardy.
\newblock On {D}irichlet’s divisor problem.
\newblock \emph{Proc. London Math. Soc. (2)}, {\bf 15}\,\penalty0 (1):\penalty0
  \,1--25, 1916.
\newblock \doi{10.1112/plms/s2-15.1.1}.

\bibitem[Huxley(2003)]{Huxley2003}
M.~N. Huxley.
\newblock Exponential {S}ums and {L}attice {P}oints {III}.
\newblock \emph{Proc. London Math. Soc. (3)}, {\bf87}\,\penalty0 (03):\penalty0
  \,591--609, 2003.
\newblock \doi{10.1112/s0024611503014485}.

\bibitem[Lamzouri(2025)]{Lamzouri2025}
Y~Lamzouri.
\newblock On the distribution of the error terms in the divisor and circle
  problems.
\newblock \emph{Int. Math. Res. Not. IMRN}, {\bf 2025}\,\penalty0 (3):\penalty0
  \,1--17, 2025.
\newblock \doi{10.1093/imrn/rnaf006}.

\bibitem[Lau and Tsang(2005)]{Lau_2005}
Y.-K. Lau and K.-M. Tsang.
\newblock Omega result for the mean square of the {R}iemann zeta function.
\newblock \emph{Manuscr. Math.}, {\bf117}\,\penalty0 (3):\penalty0 \,373--381,
  2005.
\newblock \doi{10.1007/s00229-005-0565-2}.

\bibitem[Soundararajan(2003)]{Soundararajan_2003}
K.~Soundararajan.
\newblock Omega results for the divisor and circle problems.
\newblock \emph{Int. Math. Res. Not. IMRN}, {\bf 2003}\,\penalty0
  (36):\penalty0 \,1987--1998, 2003.
\newblock \doi{10.1155/s1073792803130309}.

\bibitem[Soundararajan(2008)]{Soundararajan2008}
K.~Soundararajan.
\newblock Extreme values of zeta and {L}-functions.
\newblock \emph{Math. Ann.}, {\bf 342}\,\penalty0 (2):\penalty0 \,467--486,
  2008.
\newblock \doi{10.1007/s00208-008-0243-2}.

\bibitem[Tenenbaum(2015)]{Tenenbaum2015}
G.~Tenenbaum.
\newblock \emph{Introduction to Analytic and Probabilistic Number Theory}.
\newblock American Mathematical Society, 2015.
\newblock \doi{10.1090/gsm/163}.

\bibitem[Titchmarsh(1986)]{Titchmarsh1986}
E.~C. Titchmarsh.
\newblock The theory of the {R}iemann zeta-function. 2nd ed., rev. by {D}. {R}.
  {Heath}-{Brown}.
\newblock Oxford {Science} {Publications}. {Oxford}: {Clarendon} {Press}. x,
  412 pp., 1986.

\bibitem[Voronin(1988)]{Voronin1988}
S.~M. Voronin.
\newblock On estimates from below in the theory of the {R}iemann zeta-function.
\newblock \emph{Izv. Akad. Nauk SSSR Ser. Mat.}, {\bf52}\,\penalty0
  (4):\penalty0 \,882--892, 1988.

\end{thebibliography}
	\end{document}